\documentclass[a4paper]{amsart}

\usepackage[margin=3cm]{geometry}
\usepackage{amssymb,amsmath,amsthm,enumerate, hyperref}
\usepackage[latin1,utf8]{inputenc}
\usepackage{color}
\input xy
\xyoption{all}

\newcommand{\mb}[1]{\mathbb{#1}} 
\newcommand{\mc}[1]{\mathcal{#1}} 

\DeclareMathOperator{\Aut}{Aut}
\DeclareMathOperator{\coker}{Coker}
\DeclareMathOperator{\Ker}{Ker}

\DeclareMathOperator{\End}{End}
\DeclareMathOperator{\ext}{Ext}
\DeclareMathOperator{\lgr}{-gr}
\DeclareMathOperator{\rgr}{gr-}
\DeclareMathOperator{\CM}{CM}

\DeclareMathOperator{\Hom}{Hom}
\DeclareMathOperator{\im}{Im}

\DeclareMathOperator{\Tr}{Tr}

\DeclareMathOperator{\lmod}{-mod}
\DeclareMathOperator{\rmod}{mod-}

\DeclareMathOperator{\soc}{Soc}

\DeclareMathOperator{\perf}{perf}

\DeclareMathOperator{\gldim}{gldim}
\DeclareMathOperator{\idim}{inj.dim}
\DeclareMathOperator{\pdim}{proj.dim}
\DeclareMathOperator{\Id}{Id}

\newcommand{\psmod}{\operatorname{-\underline{mod} }\nolimits}
\newcommand{\ismod}{\operatorname{-\overline{mod} }\nolimits}
\newcommand{\pshom}{\operatorname{\underline{Hom}}\nolimits}

\newcommand{\kdim}{\operatorname{dim_\Bbbk}\nolimits}

\newcommand{\bsm}{\begin{smallmatrix}}
\newcommand{\esm}{\end{smallmatrix}}

\theoremstyle{plain}
 \newtheorem{thm}{Theorem}[section]
 \newtheorem{thmintro}{Theorem}
 
 \newtheorem{prop}[thm]{Proposition}
 \newtheorem{lemma}[thm]{Lemma}
 \newtheorem{cor}[thm]{Corollary}
 
\theoremstyle{definition}
 \newtheorem{defin}[thm]{Definition}
 \newtheorem*{notation}{Notation}
 
\theoremstyle{remark}
 \newtheorem{rmk}[thm]{Remark}
 \newtheorem{exm}[thm]{Example}

\date{\today}

\begin{document}

\title{Stably Calabi-Yau properties of derivation quotient algebras}

\author{Gabriele Bocca} 

\address{Department of Mathematics, Uppsala Universitet, \r{A}ngstr\"{o}m laboratory, L\"{a}gerhyddsv\"{a}gen 1, 752 37 Uppsala, Sweden}

\email{gabriele.bocca@math.uu.se}

\begin{abstract} The aim of this paper is to study bimodule stably Calabi-Yau properties of derivation quotient algebras. We give the definition of a twisted stably Calabi-Yau algebra and show that every twisted derivation quotient algebra $A$ for which the associated bimodule complex gives the beginning of a bimodule resolution for $A$ is bimodule stably twisted Calabi-Yau. In this setting we give a new interpretation of some results by Yu \cite{yu}, implying that $A$ is almost Koszul of periodic type. Using the characterization of higher preprojective algebras given by Amiot and Oppermann in \cite{AmOp}, we prove that finite dimensional bigraded derivation quotient algebras with homogeneous potential and exact associated complex are higher preprojective algebras of their degree-zero subalgebra, which is Koszul and $(d-1)$-representation finite.   \end{abstract}

\maketitle


\noindent

\tableofcontents

\section*{Introduction}
Quiver with potentials have been studied by many authors and are a useful tool in many areas of mathematics as well as physics (\cite{seiberg}). In connection with represenatation theory they have played a central role in the study of Fomin and Zlevinsky cluster algebras and their categorification (\cite{birs}, \cite{dwz1}, \cite{dwz2}). They also give rise to Ginzburg dg-algebras \cite{ginzburg}, which are Calabi-Yau of dimension $3$. The Jacobian algebra of a quiver with potential is defined as the path algebra of the quiver modulo the ideal generated by formal partial derivatives of the potential, that is in turn a linear combination of cycles in the quiver. This algebra coincides with the zero-homology of the Ginzburg dg-algebra associated to the quiver with potential. There is a strong connection between Jacobian algebras and (graded) $3$-Calabi-Yau algebras. For instance it has been proved by \cite{bock} that graded $3$-Calabi-Yau algebras are Jacobian algebras of some quivers with potentials. This connection has been studied more in general in \cite{BSW}, where the authors give a more theoretical construction for higher Jacobian algebras to include non-basic algebras, calling them \emph{derivation quotient algebras}. Associated to the data of a derivation quotient algebra $A$ there is a complex of projective bimodules:
\[\mc{W}^\bullet:= \quad 0\to A\otimes W_{d}\otimes A \xrightarrow{D_{d}} A\otimes W_{d-1}\otimes A \xrightarrow{D_{d-1}}\cdots \xrightarrow{D_1} A\otimes W_0 \otimes A\to 0\]

This complex is not exact in general but it is always self-dual, under the bimodule duality $\Hom_{A^e}(-,A^e)$:
\[\Hom_{A^e}(A\otimes W_{i}\otimes A,A^e) \cong A\otimes W_{d-i}\otimes A,\]
where $d$ is the degree of the superpotential defining $A$. The above duality generalizes to a twisted self-duality for twisted derivation algebras.

Exactness of the complex $\mc{W}^\bullet$ at different degrees gives precious informations about the algebra $A$ and its bimodule projective resolution. The following results generalizes facts already known for selfinjective Jacobian algebras (\cite{HerIy}).

\begin{thmintro}[Theorem \ref{thm:self derivat alg}]
Let $A=\mc{D}_\varphi(\omega,d-k)$ be a finite dimesional twisted derivation algebra with $|\omega|=d$ and such that the associated complex $\mc{W}^\bullet$ satisfies $H^0(\mc{W}^\bullet)\cong A$ and it is exact in degree one: 
\[ 0\to A\otimes W_{d}\otimes A \xrightarrow{D_{d}} \cdots \xrightarrow{D_{2}}A\otimes W_{1}\otimes A  \xrightarrow{D_1} A\otimes W_0 \otimes A \to 0. 
\]

Then $A$ is selfinjective if and only if the complex $\mc{W}^\bullet$ is also exact at $A\otimes W_{d-1}\otimes A$.

If moreover $A$ is Frobenius with $A=\bigoplus_{i=0}^pA_p$ with $A_p\neq 0$ then $\Ker D_d\cong {}_{\varphi\eta}A\langle d+p\rangle$ where $\eta$ is a Nakayama automorphism of $A$.
\end{thmintro}

It is also proved in \cite{BSW} that $\mc{W}^\bullet$ is a subcomplex of the Koszul complex of $A$ and this leads to the following result:

\begin{thmintro}[\cite{BSW}, Theorem 6.2]\label{thm 1}
An algebra $A$ is Koszul and twisted Calabi-Yau if and only if $A$ is a twisted derivation quotient algebra such that $\mc{W}^\bullet$ is a bimodule projective resolution of $A$.
\end{thmintro}

In the case of selfinjective algebras, the \emph{stably Calabi-Yau} property has been introduced and studied for example in \cite{bial-skow},\cite{dugas} and \cite{erd-skow} in the classification of finite and tame representation type algebras. In \cite{yu} the author studies the connection between almost Koszul (Frobenius) stably Calabi-Yau algebras and twisted derivation algebras with some exactness conditions on the associated complex. In particular Yu proves that a Frobenius algebra $A$ is $(p,q)$-Koszul and if and only if it is isomorphic to a twisted derivation algebra such that the associated complex gives the first $q$ terms of a bimodule projective resoltion for $A$. Since a twist appears in this construction, the algebra $A$ is not automatically stably Calabi-Yau in this case, but conditions for this to be true are also described.

The first goal of this paper is to study twisted Calabi-Yau properties for the stable category of modules over selfinjective algebras and give an interpretation of the results from \cite{yu} under this point of view. A preliminary step in this direction is given by the following fact:

\begin{thmintro}[Theorem \ref{thm:selfinj dual}]
Let $A$ be a selfinjective $(p,q)$-Koszul algebra. Then $A_p\otimes A_q^{!*}\simeq A_0\langle p+q\rangle$ as graded left $A$-modules. In particular $A$ is of periodic type.
\end{thmintro}

This implies that the main results of \cite{yu} actually hold for Frobenius almost Koszul algebras, since such algebras are automatically of periodic type. 

Using our terminology we provide an analogue of Theorem \ref{thm 1} for Frobenius almost Koszul algebras and, since in this case the quadratic dual algebra $A^!$ is again Frobenius, it is possible to describe explicitly the twists in terms of the Nakayama automorphisms of $A$ and $A^!$.

\begin{thmintro}[Corollary \ref{coroll: frob stabCY}]
Let $A$ be a Frobenius $(p,q)$-Koszul algebra or equivalently $A\cong\mc{D}_{\varphi}(\omega,q-2)$ with $\mc{W}^\bullet$ exact apart from in degrees $0$ and $d$.  Then $A$ is twisted stably Calabi-Yau of dimension $q$ with respect to the twist $\varphi=\xi^{q+1}\bar{\beta}$ up to inner automorphism, where $\bar{\beta}$ is the automorphism induced by the Nakayama automorphism $\beta$ of $A^!$ and $\xi$ acts as the multiplication by $(-1)^m$ on $A_m$.
\end{thmintro}

Stably Calabi-Yau properties have also appeared in the work of Amiot and Oppermann in connection with homological characterization of higher preperojective algebras (\cite{AmOp}). One of the main results in \cite{AmOp} is the following:

\begin{thmintro}[\cite{AmOp}, Theorems 4.8 and 5.9]
The map $A\mapsto\Pi_d(A)$ gives a one to one correspondence between $(d-1)$-representation finite algebras and finite dimensional selfinjective graded algebras $\Pi$ satisfying $\mathbf{R}\Hom_{\Pi^e}(\Pi,\Pi^e)[d+1]\cong\Pi\langle 1\rangle$ in $\mc{D}^b(\Pi^e\lgr)/\rgr\perf\Pi^e$.
\end{thmintro}

Note that a similar correspondence in the infinite representation type case has been shown in \cite{hio}.

In general, Gorenstein algebras satisfying the condition:
\begin{equation}\label{eq: 0}
\mathbf{R}\Hom_{\Pi^e}(\Pi,\Pi^e)[d+1]\cong\Pi\langle p\rangle 
\end{equation}
are called \emph{bimodule stably Calabi-Yau of Gorenstein parameter p} and this condition implies the classical stably Calabi-Yau property.  

For $d=2,3$ one implication of the above result holds in greater generality and this has been proved using the description of higher preprojective algebras as derivation algebras (\cite[Theorem 6.12]{kellerCY}). If $A$ is $\tau_{d-1}$-finite of global dimension $\leq d-1$, then its $d$-preprojective algebra is bimodule stably Calabi-Yau. This results motivated the introduction of a twist in the bimodule stably Calabi-Yau duality (\ref{eq: 0}) and the following Theorem, that is the second main result of this paper:

\begin{thmintro}[Theorem \ref{thm:main}]
Let $A=\mc{D}_\varphi(\omega,d-2)$ be a derivation algebra with $|\omega|=d$ and twist $\varphi$ such that the associated complex $\mc{W}^\bullet$ is exact everywhere apart from $H^0(\mc{W}^\bullet ) \cong A$ and in position $d$. Suppose moreover that $A$ is Gorenstein of dimension $<\dfrac{d}{2}$. Then $A$ is bimodule stably $\varphi$-twisted $d$-Calabi-Yau of Gorenstein parameter $d$.
\end{thmintro}

As in the classical non-twisted setting, we prove that twisted bimodule stably Calabi-Yau algebras are in particular twisted stably Calabi-Yau, as described in Section \ref{sect:stab CY}. 

In the last part of the paper we focus on the relation between derivation algebras and higher preprojactive algebras. Recent results by Grant and Iyama (\cite{GI}) show a connection between higher preoprojective algebras of Koszul $d$-hereditary algebras and derivation algebras. The following is one of the main results of \cite{GI}:

\begin{thmintro}[\cite{GI}, Corollary 4.3]\label{thm: GI1}
Let $\Lambda=\Bbbk Q/(R)$ be a Koszul, $d$-hereditary algebra and let $\Pi$ be its $d+1$-preprojective algebra. Then 
\[\Pi \cong \dfrac{\Bbbk\overline{Q}}{\partial_p\omega} \]
where $\overline{Q}$ is a quiver obtained from $Q$ by adding some new arrows and the relations are obtained by taking formal derivatives of a superpotential $\omega$ of degree $d+1$ with respect to paths in $\overline{Q}$ of length $d-1$. If moreover $\Lambda$ is $d$-representation finite, then $\Pi$ is selfinjective and almost Koszul and simple $\Pi$-modules have periodic projective resolutions.
\end{thmintro}

Analogously to the $2$-representation finite case (\cite{HerIy}), we introduce another grading on a derivation algebra $A$ such that the superpotential is homogeneous of degree one with respect to this new grading. This generalizes the existence of a \emph{cut} for quivers with potential. Under the usual exactness conditions on the associated coomplex $\mc{W}^\bullet$, we show that the degree-zero subalgebra of $A$ is Koszul and $(d-1)$-representation finite. Moreover we show that the derivation algebra $A$ is isomorphic to the higher preprojective algebra of its degree-zero part, giving the following result:

\begin{thmintro}[Theorem \ref{thm: deriv alg preproj}]
Let $A=\mc{D}(\omega,d-2)$ be a finite dimensional $\mb{Z}^2$-graded derivation algebra such that $|\omega|=(d,1)$ and such that $\mc{W}^\bullet$ is exact except in degree zero, where $H^0(\mc{W}^\bullet)\cong A$, and in degree $d$. Let $\Lambda:=(A)_{-,0}=\bigoplus_{i\geq 0}A_{i,0}$ be the $V$-degree zero subalgebra of $A$. Then $\Lambda$ is Koszul, $(d-1)$-representation finite and we have an isomorphism $\Pi_d(\Lambda)\cong A$ of graded algebras.
\end{thmintro}

The structure of the paper is the following. 

In Section \ref{sec: prelim} we recall graded structures on algebras and the construction of derivation algebras from \cite{BSW}. In particular we focus on the definition and properties of the complex $\mc{W}^\bullet$ that plays a central role in our work. Then we recall classical results about selfinjective and Frobenius (graded) algebras, explaining the connection between such properties and partial exactness of $\mc{W}^\bullet$ in the case of derivation algebras. In the last part of Section 1 we recall the definition of Koszul and almost Koszul algebras, with particular attention on Frobenius almost Koszul algebras.

In Section \ref{sect:stab CY} we give the definition of twisted stably Calabi-Yau (selfinjective) algebras and establish their connection with almost Koszul algebras. This results give the analogous of \cite[Theorem 6.8]{BSGkoszul} in the Frobenius case.

In Section \ref{sec :bimod} we introduce a $\mb{Z}^2$-graded structure on derivation algebras and, after recalling some facts about Cohen-Macaulay modules, we give the definition of twisted bimodule stably Calabi-Yau algebras. Then we prove that Gorenstein derivation algebras such that the associated complex is exact, apart from in the first and last degree, are twisted bimodule stably Calabi-Yau.

In Section \ref{sec: preproj} we recall the definition and cassical results about higher preprojective algebras and their bigraded structure, then we prove that derivation algebras satisfying the usual partial exactness property of the associated bimodule complex are isomorphic to higher preprojective algebras of Koszul, higher representation finite algebras. 

\textbf{Acknowledgements} The author would like to thank Martin Herschend for the guidance and help provided during the preparation of this article. Special thanks go to Vanessa Miemietz and Joseph Grant for their support and feedback.  He also acknowledges the financial support from the Department of Mathematics of Uppsala University for the academic years 2018-2019 and 2019-2020, during which this article has been written.

\section{Preliminary results}\label{sec: prelim}

\subsection{Graded algebras}
Throughout the paper, we fix an algebraically closed field $\Bbbk$ . For an algebra $A$ we denote by $\rmod A$ (resp. $A\lmod$) the category of finite dimensional right (resp. left) modules over $A$. We denote by $A^e$ the enveloping algebra $A\otimes_\Bbbk A^{op}$ so that $A$-$A$-bimodules (with central $\Bbbk$-action) can be viewed as left $A^e$-modules. For a left $A$-module $M$ and an algebra automorphism $\alpha:A\to A$, $_\alpha M$ denotes the left module with action twisted by $\alpha$, that is $a\cdot m=\alpha(a)m$ for any $a\in A$ and $m\in M$. Similarly, for a right $A$-module $N$, we denote by $N_\alpha$ the module $N$ with a right twisted action $n\cdot a=n\alpha(a)$. Note that with this notation we have ${}_\alpha A\otimes_A M={}_\alpha M$, $N\otimes_A{}_\alpha A=N_{\alpha^{-1}}$ and the isomorphism $M_\alpha\cong{}_{\alpha^{-1}}M$.

A \emph{positively graded} algebra is a $\mb{Z}$-graded algebra $A=\bigoplus_{n\geq 0}A_n$ such that $A_0$ is semisimple and $A_iA_j\subseteq A_{i+j}$. Unless stated otherwise, we will always assume that each homogeneous space $A_n$ is finite dimensional over $\Bbbk$. If $a\in A$ is an homogeneous element we denote its degree by $|a|$. A \emph{graded} $A$-module $M=\bigoplus_{n\in\mb{Z}}M_n$ is an $A$-module such that $A_iM_j\subseteq M_{i+j}$ for any $i,j\in\mb{Z}$. We denote by $M\langle l\rangle$ the degree shift of $M$ by $l$, that is $M\langle l\rangle_i=M_{i-l}$. A morphism of $A$-modules $f:M\to N$ is \emph{in degree} $j$ if $f(M_i)\subseteq N_{i+j}$. We write $\hom_A(M,N)^i$ for the space of graded morphisms in dgeree $i$ and we denote by $A\lgr$, resp. $\rgr A$, the categories of graded finite dimensional left, resp. right, modules over $A$ with morphisms in degree zero. Since we will consider only finite dimensional modules, the space $\Hom_A(M,N)=\bigoplus_{i\in\mb{Z}}\hom_A(M,N)^i$ coincides with the space of all $A$-module morphisms between $M, N\in A\lmod$.

\subsection{Derivation algebras} \label{section: deriv alg}
The definitions and constructions of this section are taken from Section 2 of \cite{BSW}.

Let $S$ be a semisimple algebra over $\Bbbk$ and $V$ be a finite dimensional $S$-$S$-bimodule. We denote by $T_S(V)$ the tensor algebra of $V$ over $S$. Any tensor algebra has a non-negative graded structure obtained by setting $S$ in degree zero and $V$ in degree one. There are several ways of obtaining $S$-$S$-bimodules dual to $V$:
\begin{enumerate} \label{enum: duals}
\item $V^*:=\Hom_\Bbbk(V,\Bbbk)$ with bimodule action $(s\cdot f\cdot t)(v)=f(tvs)$;
\item $V^{*r}:=\Hom_S(V_S,S_S)$ with bimodule action $(s\cdot f\cdot t)(v)=sf(tv)$;
\item $V^{*l}:=\Hom_S({}_SV,{}_SS)$ with bimodule action $(s\cdot f\cdot t)(v)=f(vs)t$. 
\item $V^{*b}:=\Hom_{S^e}({}_{S^e}V,S^e)$ where we denote $f\in V^{*b}$ by $f_1\otimes_\Bbbk f_2$, with bimodule action $(s\cdot f\cdot t)_1(v)\otimes_\Bbbk(s\cdot f\cdot t)_2(v)=f_1(v)t\otimes_\Bbbk sf_2(v)$.
\end{enumerate}

As explained in \cite[Section 2.1]{BSW}, the choice of a non-degenerate trace function $\Tr:S\to \Bbbk$ allows us to define natural isomorphisms between the three functors above. From now on we fix such a trace function and we will identify these three dualities, denoting the dual of $V$ by $V^*$. For any $f\in V^*$ and $v\in V$ there is a canonical pairing defined by evaluation:
\[ [fv]=[vf]=f(v). \]

Denote by $\otimes $ the tensor product over the semisimple algebra $S$. The canonical pairing can be extended to a pairing on the tensor products $(V^*)^{\otimes k}\otimes_\Bbbk V^{\otimes l}\to V^{\otimes l-k}$ (for $l\geq k$) by
\[ [(f_1\otimes\cdots\otimes f_k) (v_1\otimes\cdots\otimes v_l)]=[f_1[\cdots[f_kv_1]\cdots] v_{k}]v_{k+1}\otimes\cdots \otimes v_1 \] 
and similarly for $V^{\otimes l}\otimes_\Bbbk(V^*)^{\otimes k} \to V^{\otimes l-k}$. If $l<k$ we can replace $V^{\otimes l-k}$ with $V^{*\otimes k-l}$. Note that these brackets satisfy associativity relations, that is:
\[ [(f\otimes g)v]=[f[gv]] \quad \text{and} \quad [[fv]g]=[f[vg]] \]
for any $f\in V^{*\otimes k}$, $g\in V^{*\otimes l}$ and $v\in V^{\otimes n}$, $n\geq k+l$.

\begin{defin}\label{defin:superpotential}
A \emph{potential} of degree $d$ is an homogeneous element of degree $d$ of the tensor algebra $T_S(V)$ that commutes with the $S$ action:
\[ \omega\in V^{\otimes d} \quad \text{such that} \quad \forall s\in S: s\omega=\omega s. \]

A \emph{superpotential} is a potential such that
\[ [f\omega]=(-1)^{d-1}[\omega f], \]
for any $f\in V^*$.

Let $\varphi$ be a graded $\Bbbk$-algebra automorphism of $T_S(V)$ leaving the trace function invariant. A \emph{twisted potential} of degree $d$ is a potential $\omega$ such that
\[  \forall s\in S: \varphi(s)\omega=\omega s.\]

A twisted  potential $\omega$ is called a \emph{twisted superpotential} if
\[ [\varphi^*(f)\omega]=(-1)^{d-1}[\omega f], \]
for any $f\in V^*$, where $\varphi^*$ is the dual automorphism of $\varphi$: $\varphi^*(f)=f\circ\varphi$.
\end{defin}

Given a twisted potential $\omega$ of degree $d$, for any $k\leq d$ we define maps:
\[\Delta^\omega_k:(V^{\otimes k})^*\otimes S_\varphi\to V^{\otimes d-k}\] 
such that $\Delta_k^\omega(f\otimes s)=[f\omega s]=[f\varphi(s)\omega]$ and we denote by $W_{d-k}\subseteq V^{\otimes d-k}$ the image of $\Delta_k^\omega$. 

\begin{lemma}\label{lemma: W twisted}
Let $\omega$ be a twisted potential of degree $d$ with twist $\varphi$. For any $k\leq d$ the right $S$-action on $W_{d-k}$ is twisted by $\varphi$. 
\end{lemma}
\begin{proof}
Using the properties of twisted potentials and the definition of $W_{d-k}$ we can compute the right $S$-action:
\[ [f\omega t]\cdot s=\Delta^\omega_k(f\otimes t)\cdot s=\Delta^\omega_k(f\otimes t\cdot s)=\Delta^\omega_k(f\otimes t\varphi(s))=[f\omega t\varphi(s)], \]
for any $f\in (V^{\otimes k})^*$ and any $t,s\in S$.
\end{proof}

\begin{defin}\label{defin: deriv alg}
Let $\omega$ be a $\varphi$-twisted superpotential of degree $d$. The \emph{derivation quotient algebra} (or simply \emph{derivation algebra}) of $\omega$ is defined as follows:
\[ \mc{D}_\varphi(\omega,k):=T_S(V)/\langle W_{d-k}\rangle \]
for any $k\leq d$.
\end{defin} 

\begin{rmk}
Given a finite quiver $(Q_0,Q_1)$ with set of vertices $Q_0$ and set of arrows $Q_1$ we define the \emph{source} and \emph{target} functions $s,t:Q_1\to Q_0$ such that if $i\xrightarrow{a}j$ is any arrow in $Q_1$ we have $s(a)=i$ and $t(a)=j$. We write the composition of arrows from right to left so that $i\xrightarrow{b}j\xrightarrow{a}k=ab$. The set of trivial paths $\{e_i : i\in Q_0 \}$ is a complete set of primitive idempotents of the pah algebra $\Bbbk Q$ that is naturally a graded algebra with grading given by path length. $\Bbbk Q$ is moreover a tensor algebra of the vector space $V=\Bbbk Q_1$, generated by the arrows, over $S=\Bbbk Q_0$. Here the tensor grading coincides with the radical grading. Let $\{ a\in Q_1\}$ be a basis for $\Bbbk Q_1$ and $\{ a^* :a\in Q_1\}$ be the dual basis of $(\Bbbk Q_1)^*$. Then the brackets have the following form:
\[ [a^*b]=\delta_{ab}e_{s(b)} \quad \text{and} \quad [ba^*]=\delta_{ab}e_{t(b)}.  \]

In general, if $p$ and $q$ are paths in $Q$, the bracket $[p^*q]$ correspond to the formal partial derivative:
\[ \partial_p(q)=[p^*q]. \]

A \emph{potential} of degree $d$ is an element $\omega\in\Bbbk Q_d$ that commutes with the $\Bbbk Q_0$ action, i.e. $\omega$ is a linear combination of closed paths. $\omega$ is a \emph{superpotential} if $[a^*\omega]=(-1)^{d-1}[\omega a^*]$. 

Let $\varphi$ be an automorphism of the tensor algebra $\Bbbk Q$. A \emph{twisted potential} (of degree $d$) is a linear combination $\omega$ of paths $p\in\Bbbk Q_d$ satisfying $t(p)=\varphi(s(p))$. $\omega$ is a \emph{twisted superpotential} if $[\varphi(a)^*\omega]=(-1)^{d-1}[\omega a^*]$. 
\end{rmk}

Let $A=\mc{D}_\varphi(\omega,k)$ be a derivation algebra with a $\varphi$-twisted superpotential $\omega$ of degree $d=k+2$, so that $A$ is quadratic. Note that from the $S$-bimodules $W_i$ we can obtain graded $A$-bimodules $ A\otimes W_{i}\otimes A$ by tensoring with $A$ over $S$. 
We can associate to $A$ the following complex $\mc{W}^\bullet$ of graded projective $A$-bimodules:
\begin{equation} \label{eq:W}
\mc{W}^\bullet:= \quad 0\to A\otimes W_{d}\otimes A \xrightarrow{D_{d}} A\otimes W_{d-1}\otimes A \xrightarrow{D_{d-1}}\cdots \xrightarrow{D_1} A\otimes W_0 \otimes A\to 0,  
\end{equation}
with degree-zero differentials
\[ D_i(a\otimes w_1\ldots w_i\otimes b)=\epsilon_i(aw_1\otimes w_2\ldots w_i\otimes b +(-1)^{i}a\otimes w_1\ldots w_{i-1}\otimes w_ib), \]
where 
\[ \epsilon_i:=\begin{cases} (-1)^{i(d-i)}, & \text{if } i<(d +1)/2 \\ 1, & \text{otherwise} \end{cases} \]
for each $a,b\in A$, $w_1,\ldots,w_i\in W_i$ and $1\leq i\leq d$.

\begin{lemma}[\cite{BSW}, Lemma 6.4]\label{lemma:selfdual complex}
The complex $\mc{W}^\bullet$ is a complex of graded projective $A$-bimodules. Moreover $\mc{W}^\bullet$ is twisted-selfdual, i.e. there are isomorphisms $ (W_{d-i})_\varphi\cong (W_i)^*$ inducing a duality of $A$-bimodules $\Hom_{A^e\lgr}(A\otimes W_{i}\otimes A,A^e) \cong A\otimes W_{d-i}\otimes A_\varphi\langle-d\rangle$ such that $\Hom_{A^e\lgr}(D_i,A^e)= D_{d+1-i}\langle-d\rangle$. 
\end{lemma}

\begin{proof}
This is the twisted version of Lemma 6.4, \cite[Section 6]{BSW}. A more detailed proof in the case of twisted derivation algebras can be found in \cite{yu}, in the proofs of Theorem 3.4 and Theorem 3.7.  
\end{proof}

\begin{rmk}
For future reference we recall here that the duality $(W_i)^*\cong(W_{d-i})_\varphi$ is induced by the following perfect pairing:
\[ [\phantom{a} ,\phantom{i} ] :(W_{d-i})^*\otimes (W_i)^*\to S,\quad [ f,g]:=[(f\otimes g)\omega], \]
satisfying $[f,g]=(-1)^{|f||g|}[g,\varphi^*(f)]$ where $\varphi^*$ is the graded authomorphism of $A^*$ induced by $\varphi$.
\end{rmk}

Applying the functor $F=-\otimes_AS$ to the complex $\mc{W}^\bullet$ and forgetting the right $S$-action we obtain the following complex of projective modules in $A\lgr$:

\begin{equation} \label{eq:Wleft}
 \mc{W}^\bullet\otimes_AS:= \quad 0\to A\otimes W_{d} \xrightarrow{D_{d}} A\otimes W_{d-1} \xrightarrow{D_{d-1}}\cdots \xrightarrow{D_1} A\otimes W_0\to 0,  
\end{equation}
with differential $d_i(a\otimes w_1\ldots w_i)=\epsilon_i(aw_1\otimes w_2\ldots w_i)$. In the same way, by applying $G=S\otimes_A-$ to (\ref{eq:W}) we get a complex of graded projective right modules $S\otimes_A\mc{W}^\bullet$. The following lemma is essentially \cite[Proposition 7.5]{broom} since the arguments of the original proof apply to derivation algebras. 

\begin{lemma}\label{lemma:one sided complexes}
Suppose that the two-sided complex $\mc{W}^\bullet$ (\ref{eq:W}) satisfies $H^0(\mc{W}^\bullet)\cong A$ and let $n<|\omega|=d$. Then the following are equivalent:
\begin{enumerate}
\item $\mc{W}^\bullet$ is exact at $A\otimes W_j\otimes A$ for all $0\leq j\leq n $.
\item $\mc{W}^\bullet\otimes_AS$ is exact at $A\otimes W_j$ for all $0\leq j\leq n $. 
\item $S\otimes_A\mc{W}^\bullet$ is exact at $W_j\otimes A$ for all $0\leq j\leq n $. 
\end{enumerate}
\end{lemma}
\begin{proof}
The equivalence between $(1)$ and $(3)$ in the case of quiver with potential of degree $3$ is \cite[Proposition 7.5]{broom} and the argument used in the proof applies in our setting. $(1)\Leftrightarrow(2)$ is proved in the same way.
\end{proof}

\subsection{Frobenius algebras}

For a finite dimensional algebra $A$ denote by $\mc{S}(A)$ the set of isomorphism classes of simple modules and for any simple module $T$ let $P(T)$ be its projective cover and $E(T)$ its injective envelope. The algebra $A$ is called \emph{selfinjective} if the regular module $A\in A\lmod$ is injective. The \emph{Nakayama functor} of $A$ is the functor $\nu:A\lmod \to A\lmod$, $\nu(M):=\Hom_A(M,A)^*\simeq A^*\otimes_A M$. $A$ is called \emph{Frobenius} if $A\simeq A^*$ as left $A$-modules. Equivalently, if there exists a non-degenerate bilinear form $(-,-):A\times A\to \Bbbk$ satisfying $(ab,c)=(a,bc)$ for all $a,b,c\in A$. In this case the assigment $a\mapsto (a,-)$ gives the isomorphism of left $A$-modules $A\simeq A^*$. If $A$ is Frobenius, using the non-degenrate form $(-,-)$ we can define an automorphism of $A$, $\eta:A\to A$ such that for any $a\in A$, $\eta(a)$ is the unique element satisfying $(a,-)=(-,\eta(a))$. The automorphism $\eta$ is unique up to inner automorphism and it induces an isomorphism of bimodules $A^*\simeq A_\eta$ called a \emph{Nakayama authomorphism} of $A$.  If $M$ is any left $A$-module and $T$ any simple left $A$-module, we have:
\[\nu(M)= A_\eta \otimes_A M={}_{\eta^{-1}} M \quad\text{and}\quad \nu(T)={}_{\eta^{-1}} T, \]
so that $\nu$ induces a permutation on $\mc{S}(A)$ called the \emph{Nakayama permutation} of $A$. For any basic algebra $A$, selfinjectivity is equivalent to requiring that the assignment $\mc{S}(A)\to \mc{S}(A)$, $[S]\mapsto [\soc P(S)]$ is a bijection, hence giving the Nakayama permutation. 

Since $A^*$ is injective as left $A$-module, any Frobenius algebra is selfinjective. The converse is also true when we deal with basic algebras.  

\begin{thm}\label{thm: selfinj}
Let $A$ be a selfinjective algebra. If $\kdim\soc P(T)=\kdim T$ for any $T\in\mc{S}(A)$ then $A$ is Frobenius.
\end{thm}

\begin{proof}
This is one implication of \cite[Theorem 3]{Far2}. We include the proof here for the convenience of the reader.

The Nakayama functor induces a permutation $\nu:\mc{S}(A)\to\mc{S}(A)$ and therefore an injection $\End_A(T)\hookrightarrow\End_A(\nu(T))$. Iterating we get the isomorphism $\End_A(T)\simeq\End_A(\nu(T))$ for any $T\in\mc{S}(A)$.

Writing $A=\bigoplus_{T\in\mc{S}(A)}n_TP(T)$ and then applying $\Hom_A(-,T)$ we get the following equation:
\begin{equation}\label{eq:dim simpl}
n_T=\dfrac{\kdim T}{\kdim \End_A(T)}=\dfrac{\kdim \nu(T)}{\kdim \End_A(\nu(T))}=n_{\nu(T)}.
\end{equation}

Since we also have $P(T)\simeq E(\nu(T))$ we get
\[(A_A)^*\simeq \nu(A)\simeq\bigoplus_{T\in\mc{S}(A)}n_TE(T)\simeq\bigoplus_{T\in\mc{S}(A)}n_TP(T)\simeq A \]
therefore $A$ is Frobenius.
\end{proof}

\begin{cor}\label{cor: frob is self}
Let $A$ be a basic algebra. Then $A$ is Frobenius if and only if it is selfinjective.
\end{cor}

\begin{proof}
Since $A$ is basic we have $n_T=1$ for any $T\in\mc{S}(A)$ in the proof of Theorem \ref{thm: selfinj}. This implies $\kdim T=\kdim\nu(T)$ so $A$ is Frobenius by Theorem \ref{thm: selfinj}.
\end{proof}

Recall that a finite dimensional $\Bbbk$-algebra is called \emph{periodic} if $\Omega^n_{A^e}(A)\simeq A$ as $A^e$-modules, for some $n>0$. The smallest positive integer $n$ giving the isomorphism is called the \emph{period} of $A$. The fact that periodic algebras are selfinjective was first stated by M. C. R. Butler. The following theorem is a bit more general and can be found in \cite[Theorem 1.4]{GSS}. 

\begin{thm}\label{thm: periodic selfinj}
Let $A$ be an indecomposable finite dimensional algebra over an algebraically closed field $\Bbbk$. Then the following are equivalent:
\begin{enumerate}
\item There exists some $n>0$ such that $\Omega^n_A(T)\cong T$ for every simple module $T$ in $A\lmod$.
\item There exists an algebra automorphism $\sigma\in\Aut (A)$ such that $\Omega^n_{A^e}(A)\simeq{}_\sigma A$ and $\sigma(e_i)=e_i$ for a complete set of orthogonal idempotents $\{e_i\}_I$.
\end{enumerate}
Moreover under the above equivalent conditions the algebra $A$ is selfinjective.
\end{thm}

\begin{proof}
We refer to Theorem 1.4 in \cite{GSS} for the proof of the equivalence between statements (1) and (2). We recall here the argument showing that $A$ is selfinjective.

Suppose that $\Omega^n_{A^e}(A)\simeq{}_\sigma A_1$ for some $n>0$ and let $P^{n-1}$ be the $(n-1)$-th term of a bimodule projective resolution of $A$. Since $\Hom_A({}_\sigma A,A)\simeq A_A$ as right $A$-modules and $\Hom_A({}_\sigma A,A)^*\simeq {}_\sigma A\otimes_A A^*$ as left $A$-modules, we have the following isomorphisms of left $A$-modules:
\[A^*\simeq\Hom_A({}_\sigma A,A)^*\simeq {}_\sigma A\otimes_A A^*\simeq \Omega^n_{A^e}(A)\otimes_A A^*\simeq \Omega^n_A(A^*). \]

Therefore the injective module $A^*$ is isomorphic to a submodule of a projective module and hence it is projective. It follows by definition that $A$ is selfinjective.
\end{proof}

Since we will mainly be dealing with graded algebras we recall the following result from \cite{yu}.

\begin{lemma}\cite[Lemma 2.4]{yu}\label{lemma: graded frob}
Leta $A=\bigoplus_{i=0}^pA_i$ be a finite dimensional graded algebra with $A_0$ sempisimple and $A_p\neq 0$. Then the following are equivalent:
\begin{enumerate}
\item $A$ is Frobenius.
\item $\kdim A_0=\kdim A_p$ and if $\{x_1,\ldots,x_n\}$ is a basis of $A_p$, then the bilinear form $(a,b):=\left(\sum_{i}x_i^*\right)(ab)$ is non-degenerate.
\item $A\langle -p\rangle\cong A^*$ in $A\lgr$.
\end{enumerate}
\end{lemma}

If $A$ is Frobenius with Nakayama automorphism $\eta$ then its enveloping algebra $A^e$ is also Frobenius (this holds more generally for every tensor product of Frobenius $\Bbbk$-algebras) with Nakayama automorphism $\bar{\eta}=\eta\otimes\eta^{-1}$ since $A^{op}$ has Nakayama automorhism $\eta^{-1}$. 

Moreover, if $A=\bigoplus_{i=0}^p$ is graded, $A^e$ is also graded with
\[ \left(A^e\right)_i=\bigoplus_{i=j+k}A_j\oplus A_k,\]
therefore $A^e=\bigoplus_{i=0}^{2p}\left(A^e\right)_i$ and $(A^e)^*\cong (A^e)_{\bar{\eta}}\langle -2p\rangle$ by Lemma \ref{lemma: graded frob}.

\begin{lemma}\label{lemma: envel frob}
Let $A=\bigoplus_{i=0}^pA_i$ be graded and Frobenius with $A_p\neq 0$ and Nakayama automorphism $\eta$. Then:
\begin{enumerate}
\item $\Hom_{A\lgr}(M,A)\cong {}_{\eta}(M^*)\langle p\rangle$ for any $M\in A\lgr$.
\item $\Hom_{A^e\lgr}(A,A^e)\cong {}_\eta A\langle p\rangle$.
\end{enumerate}
\end{lemma}

\begin{proof}
The first statement is the graded version of \cite[Proposition B.0.8]{kellerCY}:
\[ \Hom_{A\lgr}(M,A)\cong \Hom_{A\lgr}(M,A^{**})\cong \Hom_\Bbbk(A^*\otimes M, \Bbbk)\cong ({}_{\eta^{-1}}M\langle -p\rangle )^*\cong {}_{\eta}(M^*)\langle p\rangle.\]

To prove the second statement we use point (1) with $A^e$ and $A$ instead of $A$ and $M$, keeping in mind that $(A^e)^*\cong (A^e)_{\eta\otimes\eta^{-1}}\langle -2p\rangle$:
\[ \Hom_{A^e\lgr}(A,A^e)\cong {}_{\eta\otimes\eta^{-1}}(A^*)\langle 2p\rangle\cong {}_{\eta}(A^*)_{\eta^{-1}}\langle 2p\rangle \cong {}_\eta A \langle p\rangle. \]
\end{proof}

For finite dimensional twisted derivation algebras the selfinjective property is equivalent to exactness in some degrees of the associated complex $\mc{W}^\bullet$ (\ref{eq:W}). This fact was already known for Jacobian algebras of quiver with potentials (see \cite[Theorem 3.7]{HerIy}).

\begin{thm}\label{thm:self derivat alg}
Let $A=\mc{D}_\varphi(\omega,d-k)$ be a finite dimesional twisted derivation algebra with $|\omega|=d$ and such that the associated complex $\mc{W}^\bullet$ (\ref{eq:W}) satisfies $H^0(\mc{W}^\bullet)\cong A$ and it is exact in degree one: 
\begin{equation}\label{eq:W augm}
 0\to A\otimes W_{d}\otimes A \xrightarrow{D_{d}} \cdots \xrightarrow{D_{2}}A\otimes W_{1}\otimes A  \xrightarrow{D_1} A\otimes W_0 \otimes A \to 0. 
\end{equation}

Then the following are equivalent:
\begin{enumerate}
\item $A$ is selfinjective.
\item The complex (\ref{eq:W augm}) is also exact at $A\otimes W_{d-1}\otimes A$.
\end{enumerate}

If $A$ is Frobenius and $A=\bigoplus_{i=0}^pA_p$ with $A_p\neq 0$ then $\Ker D_d\cong {}_{\varphi\eta}A\langle d+p\rangle$ where $\eta$ is a Nakayama automorphism of $A$.
\end{thm}

\begin{proof}
First note that under our assumption, applying $-\otimes_AS$ to the complex $(\ref{eq:W augm})$ we obtain a complex of graded projective left $A$-modules that is exact in the first two positions, hence it gives the beginning of a graded projective resolution for the module ${}_AS$. Since $A$ is finite dimensional, it is selfinjective if and only if $\ext^1_{A\lgr}(S,A)=0$ and in turn this is equvalent to the exactness of the following complex in degree one:
\[\Hom_{A\lgr}(A\otimes W_0,A)\to \Hom_{A\lgr}(A\otimes W_1,A)\to\cdots \to \Hom_{A\lgr}(A\otimes W_{d}).\]

Since the $S$-bimodules are finitely generated we have isomorphisms:
 \[\Hom_{A\lgr}(A\otimes W_i,A)\cong W_i^*\otimes A\cong W_{d-i}\otimes A_\varphi\langle-d\rangle\] 
 or any $0\leq i\leq d$, where the first isomorphism follows from \cite[Lemma 3.2]{yu} and the second by the twisted selfduality of $\mc{W}^\bullet$ (Lemma \ref{lemma:selfdual complex}). Hence $A$ is selfinjective if and only if the complex
 \[W_{d}\otimes A_\varphi\langle-d\rangle \to W_{d-1}\otimes A_\varphi\langle-d\rangle \to \cdots \to W_0\otimes A_\varphi\langle-d\rangle \]
 is exact at $W_{d-1}\otimes A_\varphi\langle-d\rangle$. Since $\varphi$ is an automorphism, $-\otimes_AA_{\varphi}$ is an autofunctor, therefore the exactness of the above complex is equivalent to $(2)$ by Lemma \ref{lemma:one sided complexes}.
 
Assume that $A=\bigoplus_{i=0}^p A_p$ is Frobenius with Nakayama automorphism $\eta$ and $A_p\neq 0$. Then applying the functor $\Hom_{A^e}(-,A^e)$ to the complex $\mc{W}^\bullet$ and using Lemma \ref{lemma:selfdual complex} and Lemma \ref{lemma: envel frob} we get that the complex:
\[0\to {}_\eta A\langle p\rangle \to A\otimes W_{d}\otimes A_\varphi\langle-d\rangle \to  \cdots \to A\otimes W_0\otimes A_\varphi\langle-d\rangle \to A_\varphi\langle-d\rangle \to 0\]
is exact in the first and last two positions. Applying the autofunctor $-\otimes_{A}A_{\varphi^{-1}}$ and the degree-shift $\langle d\rangle$ we get $\Ker D_d\cong {}_{\varphi\eta}A\langle d+p \rangle$ as claimed.

\end{proof}

We also have the following fact about the exactness of the complex $\mc{W}^\bullet$. 


\begin{cor}\label{cor:selfinj W ex}
Let $A=\mc{D}_\varphi(\omega,d-2)$ be as in Theorem \ref{thm:self derivat alg}. If $\mc{W}^\bullet$ is exact in degree $i$ for $1\leq i \leq n<d$ then $\mc{W}^\bullet$ is also exact in degree $j$ for $(n-d)\leq i <d$.
\end{cor}

\begin{proof}
Under the hypotesis the algebra $A$ is selfinjective so $\ext^i_A(-,A)=0$ for any $i>0$. Therefore the claim is proved using the same argument of the proof of Theorem \ref{thm:self derivat alg}.
\end{proof}

\begin{rmk}
A consequence of Corollary \ref{cor:selfinj W ex} is that if $\mc{W}^\bullet$ is exact in position $i$ for $0<i\leq d/2+1$ then it is exact everywhere apart from position $0$ and $d$. 
\end{rmk}

\subsection{Almost Koszul algebras}

Let $V$ be an $S$-$S$-bimodule, $S$ semisimple, and let $T_S(V)$ be the tensor algebra of $V$ over $S$. If $R$ is a subspace of $V\otimes V$, the algebra $A=T(V)/\langle R \rangle$ is called a \emph{quadratic} algebra. For any quadratic algebra $A$ its \emph{quadratic dual} is defined as $A^!=T(V^*)/\langle R^\perp \rangle$ where $R^\perp$ is the space orthogonal to $R$ in $V^*\otimes V^*$. Note that the standard duality $(-)^*=\Hom_\Bbbk(-,\Bbbk)$ extends to a duality of $S$-$S$-bimodules, as explained at the beginning of Section \ref{section: deriv alg}.

The definition of almost Koszul algebra was first introduced by Brenner, Butler and King in \cite[Definition 3.1]{bbk}.

\begin{defin} \label{defin: alm kosz}
A graded algebra $A=\bigoplus_{n\geq0}A_n$ with semisimple degree zero part $A_0$ is called \emph{left almost Koszul} or $(p,q)$-\emph{Koszul} if there exist two integers $p,q>0$ such that
\begin{enumerate}
\item $A_i=0$ for all $i> p$,
\item there is a graded exact complex of left $A$-modules:
\[ 0\to W\to P_q \to P_{q-1}\to \ldots \to P_1 \to P_0 \to A_0 \to 0  \]
where $P_i$ is projective and generated by its compnent of degree $i$ for $i=0,\ldots q$ and $W=A_p\otimes (P_q)_q$.
\end{enumerate} 

If $W=A_0\langle p+q\rangle$ then we say that $A$ is left almost Koszul \emph{of periodic type}.
\end{defin}

\begin{rmk}
We recall that under the assumptions of Definition \ref{defin: alm kosz}, the algebra $A$ is called (left) \emph{Koszul} if 
\[ \ldots \to P_1 \to P_0 \to A_0 \to 0  \]
is a graded projective resolution of $A_0$ where each $P_i$ is generated by its compnent of degree $i$.
\end{rmk}

By \cite[Prop. 3.4]{bbk} any left almost Koszul algebra is also right almost Koszul. Moreover $W$ is always semisimple and concentrated in degree $p+q$.
 
\begin{rmk} From now on, by almost Koszul algebra we will mean a $(p,q)$-algebra such that $p,q\geq 2$. Thanks to \cite[Prop. 3.7]{bbk}, this condition guarantees that $A$ is a quadratic algebra generated by $A_1$ as a tensor algebra over $A_0$.
\end{rmk}

As in the case of Koszul algebras, the almost Koszul property is equivalent to some conditions on the \emph{Koszul complex} of $A$. For any quadratic algebra $A=T(V)/\langle R\rangle$ and $i\geq0$ we can define a $S$-$S$-bimodule $K^i\subseteq V^{\otimes i}$ as follows:
\[ K^0=S, \quad K^1=V, \quad K^{i+1}=V\otimes K^i\cap K^i\otimes V \text{ for all } i>1. \]

The two-sided Koszul complex of $A$ is 
\[ \cdots \to A\otimes K^n\otimes A \to\cdots \to A\otimes K^1\otimes A \to A\otimes K^0\otimes A \to 0, \]
with differentials given by the difference of the following two compositions:
\[ A\otimes K^n\otimes A \to AA_1\otimes K^{n-1}\otimes A \to A\otimes K^{n-1}\otimes A, \]
\[ A\otimes K^n\otimes A \to A\otimes K^{n-1}\otimes A_1 A \to A\otimes K^{n-1}\otimes A, \]
 analogously to the differentials of the complex $\mc{W}^\bullet$ form Lemma \ref{lemma:selfdual complex}. It is well known (\cite{BSGkoszul}, Remark after Definition 2.8.1) that $(A^!_n)^*\cong K^n\subseteq V^{\otimes n}$ for $n>0$. The following characterization of Koszul algebra can be found for instance in \cite[Proposition A.2]{BrGai}.
 
\begin{thm}
$A$ is Koszul if and only if its two-sided Koszul complex is a projective resolution of $A$ as an $A^e$-module.
\end{thm}

Forgetting the right (resp. left) $A$-module structure in the two-sided Koszul complex we obtain the \emph{left} (resp. \emph{right}) \emph{Koszul complex} of the algebra $A$:
 \[ \cdots \to A\otimes K^n \to\cdots \to A\otimes K^1 \to A\otimes K^0 \to 0 .\]
 
The following lemma gives an important connection between the complex $\mc{W}^\bullet$ and the two-sided Koszul complex of a quadratic derivetion algebra $A=\mc{D}(\omega,|\omega|-2)$.

\begin{lemma} \label{lemma: Koszul subcomplex}
Let $A=\mc{D}(\omega,|\omega|-2)$ be a quadratic $\varphi$-twisted derivation algebra. Then the complex $\mc{W}^\bullet$ is a subcomplex of the two-sided Koszul complex of $A$. 
\end{lemma}

\begin{proof}
This fact is proved in Lemma 6.5 and Remark 6.6 in \cite{BSW}. The crucial observation is that in general we have inclusions $(W_k)_\varphi\subseteq (A^!_k)^*$ for $0\leq k\leq |\omega|$ with equality if $A$ is Koszul.
\end{proof}

It follows from \cite[Prop. 3.9]{bbk} that a graded algebra $A$ is left $(p,q)$-Koszul if and only if $A_n=0$ for all $n>p$, $K^m=0$ for all $m>q$ and the only non-zero homology groups of the left Koszul complex of $A$ are $A_0$ in degree zero and $A_p\otimes K^q$ in degree $p+q$.

\begin{lemma}\label{lemma: kosz compl}
For any $(p,q)$-Koszul algebra $A$ we have a graded exact sequence in $A\lgr$:
\[ 0\to A_p\otimes A_q^{!*} \to A\otimes A_q^{!*} \xrightarrow{d_q} A\otimes A_{q-1}^{!*} \to \cdots \to A\otimes A_1^{!*} \xrightarrow{d_1} A\otimes A_0^{!*}\to A_0\to 0.\]
\end{lemma}

\begin{proof}
This is a consequence of Proposition 3.9 and Remark 3.10 in \cite{bbk}. See also \cite[Lemma 2.2]{yu}.
\end{proof}

The following is a simple but interesting result on selfinjective almost Koszul algebras. 

\begin{thm} \label{thm:selfinj dual}
Let $A$ be a selfinjective $(p,q)$-Koszul algebra. Then $A_p\otimes A_q^{!*}\simeq A_0\langle p+q\rangle$ as graded left $A$-modules. In particular $A$ is of periodic type.
\end{thm}

\begin{proof}
Let $T$ be a simple $A$-module. Since $A$ is selfinjective, $\Omega_A^{q+1}(T)$ is a non-zero indecomposable direct summand of the semisimple module $W\simeq A_p\otimes A_q^{!*}$, by Lemma \ref{lemma: kosz compl}. Therefore $\Omega_A^{q+1}(T)$ is again a simple module and, since $A_0$ is finite dimensional, $\Omega_A^{q+1}(-)$ induces a bijection on the set of isomorphism classes of simple modules. Hence $\Omega^{q+1}_A(A_0)\simeq A_p\otimes A_q^{!*}\simeq A_0\langle p+q\rangle$ in $A\lgr$.
\end{proof}

Analogously to the classic Koszul setting, the quadratic dual of (quadratic) almost Koszul algebras has an interesting homological description.

\begin{prop}\label{prop: alm Kosz dual}
If $A$ is a left $(p,q)$-Koszul algebra with $p,q\geq2$, then $A^!$ is a left $(q,p)$-Koszul algebra. Moreover $A^!\simeq \bigoplus_{i=0}^q\ext^i_A(A_0,A_0)$.
\end{prop}

\begin{proof}
The first and second statement are respectively Proposition 3.11 and Remark 3.12 of \cite{bbk}.  
\end{proof}

\begin{prop}\label{prop: frob period dual}

Let $A$ be a $(p,q)$-Koszul algebra with $p,q\geq 2$. If $A$ is Frobenius then both $A$ and $A^!$ are Frobenius and of periodic type.
\end{prop}

\begin{proof}

By Theorem \ref{thm:selfinj dual} $A$ is of periodic type therefore by one implication of \cite[Prop. 3.3]{yu} $A^!$ is also Frobenius and of periodic type.
\end{proof}

We also recall the following result that is Theorem 3.7 in \cite{yu}. It is the analogous of \cite[Theorem 6.8]{BSW} in case of Frobenius algebras, establishing a conncetion between almost Koszul Frobenius algebras and stably Calabi-Yau algebras. Note that, differently from the original statement, we don't need to require that $A$ is (almost Koszul) of periodic type since this is a consequence of Theorem \ref{thm:selfinj dual}.

\begin{thm}[\cite{yu}, Theorem 3.7]\label{thm: 3.7}
Let $A=\bigoplus_{i=0}^pA_i$ be a finite dimensional Frobenius algebra. Then $A$ is $(p,q)$-Koszul if and only if $A$ is isomorphic to a derivation algebra $\mc{D}_\varphi(\omega,q-2)$ for a twisted superpotential $\omega$ of degree $q$ such that the complex $\mc{W}^\bullet$ is exact everywhere apart from in degree zero, where $H^0(\mc{W}^\bullet)\cong A$ and in degree $q$.
\end{thm}

\begin{cor}
Under the assumptions of Theorem \ref{thm:3.4}, if $A$ is  $(p,q)$-Koszul then $\Omega^{q+1}_{A^e}(A)\cong {}_{\varphi\eta}A\langle p+q\rangle $ where $\varphi$ is an automorphism such that $A\cong\mc{D}_\varphi(\omega,q-2)$.
\end{cor}

\begin{proof}
This is an immediate consequence of Theorem \ref{thm:self derivat alg}.
\end{proof}

For Frobenius almost Koszul algebras the twisting automorphism $\varphi$ can be described precisely in terms of the Nakayama automorphism of the quadratic dual algebra $A^!$. This was done in \cite[Theorem 3.4]{yu}.
 
\begin{thm}[\cite{yu}, Theorme 3.4]\label{thm:3.4}
Let $A$ be a Frobenius $(p,q)$-Koszul algebra and denote by $\xi$ be the graded automorphism defined on $A_i$ as the multiplication with $(-1)^i$. Then $\Omega^{q+1}_{A^e}(A)\cong{}_{\xi^{q+1}\bar{\beta}\eta}A\langle p+q\rangle$ in $A^e\lgr$.
\end{thm}

\section{Stably Calabi-Yau properties} \label{sect:stab CY}

In this section we recall some facts about Calabi-Yau triangulated (graded) categories. We refer to \cite[Appendix A]{bock} and \cite{van} for a more detailed discussion.  

Let $(\mc{C},\Sigma)$ be a $\Hom$-finite ($\Bbbk$-linear) triangulated category. A triangulated autoequivalence $(\mb{S},\alpha)$, where $\alpha:	\mb{S}\Sigma\to\Sigma\mb{S}$ is a natural isomorpism, is called a \emph{Serre functor} if for any $X,Y\in\mc{C}$ there is a natural isomorphism 
\[t_{Y,X}:\Hom_\mc{C}(Y,\mb{S}X)\to \Hom_\mc{C}(X,Y)^* \]
making the following diagram commute:

\[ \xymatrix{\Hom_\mc{C}(\Sigma Y, \mb{S}\Sigma X) \ar[d]^{\Hom_\mc{C}(\Sigma Y,\alpha_X)} \ar[r]^{t_{\Sigma Y,\Sigma X}} & \Hom_\mc{C}(\Sigma X,\Sigma Y)^* \ar[r]^{\Sigma^*} & \Hom_\mc{C}(X,Y)^* \ar@{=}[d] \\
 \Hom_\mc{C}(\Sigma Y, \Sigma\mb{S} X) \ar[r]^{\Sigma^{-1}} & \Hom_\mc{C}( Y,\mb{S} X) \ar[r]^{t_{Y, X}} & \Hom_\mc{C}(X, Y)^*} \]

The natural isomorphism $t_{-,-}$ identifies the Serre functor $\mb{S}$ up to natural isomorphism (\cite[Lemma I.1.3]{reiten-VDBergh}). Moreover, if a Serre functor exists, it commutes with any autoequivalence of $\mc{C}$: if $F:\mc{C}\to\mc{C}$ is an autoequivalence, then $F\mb{S}F^{-1}$ is another Serre functor, hence isomorphic to $\mb{S}$.

For any object $X$ of $\mc{C}$, setting $X=Y$ in the isomoprhism above gives a distinguished function $\Tr_X=t_{X,X}(id_X):\Hom_\mc{C}(X,\mb{S}X)\to\Bbbk$ called the \emph{trace map}. This defines a non-degenerate bilinear pairing
\[\Hom_\mc{C}(X,\mb{S}Y)\times \Hom_\mc{C}(Y,X) \to \Bbbk \]
satisfying $\Tr_X(g\circ f)=\Tr_Y(\mb{S}(f)\circ g)$. In this case the natural isomorphism $\alpha: \mb{S}\Sigma \to \Sigma\mb{S}$ making $\mb{S}$ into a triangulated functor of $(\mc{C},\Sigma)$ is uniquely determined by the following formula (see \cite[Appendix A, p.30]{bock}, Theorem A.4.4):
\[\Tr_X(\Sigma^{-1}(\alpha_X\circ f))=-\Tr_{\Sigma X}(f) \]
for any $X\in\mc{C}$ and $f:\Sigma X\to\mb{S}(\Sigma X)$. Therefore for any distinguished triangle
\[ X\xrightarrow{f} Y\xrightarrow{g} Z \xrightarrow{h}  \Sigma X \]
the triangle 
\[ \mb{S}(X) \xrightarrow{\mb{S}(f)} \mb{S}(Y)\xrightarrow{\mb{S}(g)} \mb{S}(Z) \xrightarrow{\alpha_X\circ\mb{S}(h)}  \Sigma\mb{S} (X)\]
is also distinguished.

\begin{defin} \label{defin:cy}
The category $\mc{C}$ is called \emph{weakly }$d$\emph{-Calabi-Yau} if it has a Serre functor $\mb{S}$ and $d$ is the smallest positive integer such that there is an isomorphism of functors $\mb{S}\simeq \Sigma^d$ (see \cite[Section 2.6]{kellerCY}). We say that $\mc{C}$ is \emph{strongly} $d$\emph{-Calabi-Yau} if there is an isomorphism of triangulated functors $(\mb{S},\alpha)\simeq(\Sigma^d,(-1)^d)$.
\end{defin}

\subsection{Twisted stably Calabi-Yau algebras} 
In this section we want to give the definition of \emph{twisted} stably Calabi-Yau algebras and, analogously to what has been done in \cite{BSW} in the twisted Calabi-Yau setting, interpret some results from \cite[Section 4]{yu} according to our definition. As in \cite{yu}, in this section we consider ungraded Calabi-Yau properties.

Let $A$ be a selfinjective algebra and $A\psmod$ the category of projectively stable $A$-modules (see \cite[Ch. X]{ARS}). The category $A\psmod$ is $\Hom$-finite and triangulated with suspension functor $\Sigma=\Omega_A^{-1}$, where $\Omega_A$ is the syzygy functor. Let $\nu$ be the Nakayama functor and $\tau$ the Auslander--Reiten translation . If $A$ is selfinjective we have that $\tau\simeq\nu\Omega_A^2$ as endofunctors on $A\lmod$ (\cite[Prop IV.3.7]{ARS}). By the Auslander--Reiten formula we have the following isomorphisms:
\[ \pshom_A(X,Y)\simeq\ext_A^1(Y,\tau X)^*\simeq\pshom_A(Y,\Omega_A^{-1}\tau X)^* \]
which are natural in both $X$ and $Y$. Therefore the functor $\Omega^{-1}_A\tau$ satisfies the definition of Serre functor for the triangulated category $A\psmod$. Following \cite{yu}, we say that the algebra $A$ is \emph{stably Calabi-Yau} if its stable category $A\psmod$ is weakly Calabi-Yau. Since $\tau\simeq\nu\Omega_A^2$, $A$ is stably Calabi-Yau if there is some integer $d>0$ such that $\tau\simeq \Omega^{-d+1}_A$ or equivalently $\nu\simeq\Omega^{-d-1}_A$ as functors. 

Let $\phi:A\to A$ be an algebra endomorphism of $A$ and let $\phi^*={}_\phi A\otimes_A-:A\lmod\to A\lmod$ be the corresponding endofunctor sending a left $A$-module $M$ to $\phi^*(M)={}_\phi M$. Then $\phi^*$ is an autofunctor if and only if $\phi$ is an automorphism. Moreover $\phi^*$ preserves projective modules, therefore it induces a functor on the stable module category $A\psmod$ that commutes with the cosyzygy functor $\Omega_A^{-1}$.

We give the following definition of \emph{twisted stably Calabi-Yau} algebra (see also \cite[Section 2.5]{skew} for a stronger definition in the case of triangulated multi-graded categories).

\begin{defin}\label{defin:tw st cy}
Let $A$ be a selfinjective algebra and let $\phi:A\to A$ be an algebra automorphism. We say that $A$ is $\phi$-\emph{twisted stably Calabi-Yau} of dimension $d$ if $\Omega_A^{-d}\phi^*$ is a Serre functor for $A\psmod$:
\[ \pshom_A(X,Y)\simeq\pshom_A(Y,\Omega_A^{-d}\phi^* (X))^*\simeq\pshom_A(Y,\Omega_A^{-d}({}_\phi X))^* \]

Equivalently, if there is an isomorphism of functors
\[ \Omega_A^{-1}\tau\simeq\Omega_A^{-d}\phi^* \]
for some $d>0$. 
\end{defin}

\begin{rmk}
A conditions equivalent to Definition \ref{defin:tw st cy} is that there is an isomorphism of functors $\tau\simeq\Omega_A^{-d+1}\phi^*$ or $\nu\simeq\Omega^{-d-1}_A\phi^*$.
\end{rmk}

Note that if $\phi=\Id$ then $A$ is stably Calabi-Yau in the classical sense so every stably Calabi-Yau algebra is twisted stably Calabi-Yau. The following example shows that the converse is not always true.

\begin{exm}
Let $A$ be the path algebra the following quiver:
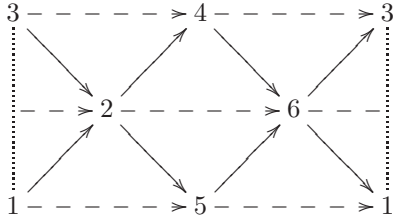
\[\xymatrix{ 3 \ar@{.}[dd] \ar[dr] \ar@{-->}[rr] & & 4 \ar[dr] \ar@{-->}[rr] & & 3\ar@{.}[dd] \\
          \ar@{-->}[r]   & 2 \ar[dr] \ar[ur] \ar@{-->}[rr] & & 6 \ar[dr] \ar[ur] \ar@{--}[r] & \\ 1 \ar[ur] \ar@{-->}[rr] & & 5  \ar[ur] \ar@{-->}[rr] & & 1}\]
(with identification along the vertical dotted lines) modulo the ideal generated by commutativity and zero-relations as indicated by the dashed arrows. The algebra $A$ is Frobenius and an easy computation shows that the Nakayama permutatin is $\nu(=\nu^{-1})=(14)(26)(35)$. Therefore, denoted with $S_i$ the left simple $A$-module generated by the idempotent $e_i$, we have $\nu(S_i)=A^*\otimes_A S_i=S_{\nu(i)}$ for any $i=1,\ldots,6$. On the other hand, the smallest integer $l>0$ such that $\Omega_A^{-l}(S_i)$ is again simple is $l=3$ and $\Omega_A^{-3}$ induces a permutation $\sigma:\mc{S}(A)\to\mc{S}(A)$ given by $\sigma=(13)(45)$. Then there can be no integer $d>0$ such that $\Omega^{-d-1}_A\simeq\nu$. However, $A$ is twisted stably Calabi-Yau with twist $\phi^*$ induced by the permutation $\phi=(15)(26)(34)$ since $\nu= (13)(45)\circ (15)(26)(34)=\sigma\circ\phi$ that implies $\nu\simeq\Omega_A^{-3}\phi^*$.
\end{exm}

\begin{prop} \label{prop: period tw cy}
Let $A$ be a basic, indecomposable, finite dimensional $\Bbbk$-algebra. If there exists an algebra automorphism $\sigma\in\Aut(A)$ and some integer $n\geq 0$ such that $\Omega^{n+1}_{A^e}(A)\cong {}_\sigma A$ as $A$-$A$-bimodules then $A$ is (Frobenius and) twisted stably Calabi-Yau with twist $\phi=\sigma\eta^{-1}$, where $\eta$ is a Nakayama automorphism of $A$.
\end{prop}

\begin{proof}
By  Theorem \ref{thm: periodic selfinj} the algebra $A$ is selfinjective and, since it is basic, it is also Frobenius. Let $\eta$ be a Nakayama automorphism for $A$, so that $A^*\cong A_\eta$. Then $\nu(M)={}_{\eta^{-1}} A\otimes_A M\cong {}_{\eta^{-1}} M$ and $\Omega^{n+1}_A(M)\cong {}_\sigma A\otimes_A M\cong{}_\sigma M$ for every non-projective left module $M$. Therefore for any non-projective modules $X,Y$ we have the following natural isomorphisms:
\[ \begin{split} \pshom_A(X,Y) & \cong \pshom_A(Y,\Omega^{-1}\tau(X))^* \\  & \cong\pshom_A(Y,\Omega^{-1}\nu\Omega^2(X))^* \\ &\cong \pshom_A(Y,\Omega({}_{\eta^{-1}} X))^* \\ & \cong\pshom_A(\Omega^n(Y),{}_{\sigma\eta^{-1}} X)^* \\  & \cong\pshom_A(Y,\Omega^{-n}\sigma^*\nu( X))^*\end{split} \]
where $\sigma^*(-)={}_\sigma A\otimes_A-$. By uniqueness of the Serre functor for $A\psmod$ there is an isomorphism $\Omega_A^{-n}\sigma^*\nu\cong \Omega^{-1}\tau$ so that $A$ is the twisted stably Calabi-Yau with twist $\phi=\sigma\eta^{-1}$ (up to inner automorphisms).
\end{proof}

\begin{rmk}
Note that another choice of Nakayama automorphism $\eta'=a^{-1}\eta a$, $a\in A$, doesn't change the twisted Calabi-Yau dimension but gives another twist $\phi'=\sigma\eta'$ that differs from $\phi$ by an inner automorphism.
\end{rmk}

\begin{cor}\label{coroll: frob stabCY}
Let $A$ be a Frobenius $(p,q)$-Koszul algebra or equivalently $A\cong\mc{D}_{\varphi}(\omega,q-2)$ with $\mc{W}^\bullet$ exact apart from in degrees $0$ and $d$.  Then $A$ is twisted stably Calabi-Yau of dimension $q$ with respect to the twist $\varphi=\xi^{q+1}\bar{\beta}$ up to inner automorphism, where $\bar{\beta}$ is the automorphism induced by the Nakayama automorphism $\beta$ of $A^!$ and $\xi$ acts as the multiplication by $(-1)^m$ on $A_m$.
\end{cor}


\begin{proof}
By Proposition \ref{prop: frob period dual} $A^!$ is also Frobenius and let $\beta$ be a Nakayama automorphism for $A^!$. By Theorem \ref{thm:3.4} we have the following isomorphism of $A$-$A$-bimodules: $\Omega^{q+1}_{A^e}(A)\simeq {}_{\xi^{q+1}\bar{\beta}\eta}A$. Therefore by Proposition \ref{prop: period tw cy} $A$ is twisted stably Calabi-Yau of dimension $q+1$ and twist $\phi=\xi^{q+1}\bar{\beta}\eta\eta^{-1}=\xi^{q+1}\bar{\beta}$.
\end{proof}

Let $A$ be a $(p,q)$-Koszul algebra and define $W=A_q\otimes A_q^{!*}$, $W_0=S$ and $W_n=W^{\otimes n}$ for any $n>0$. For any $n\geq 0$ we have the following (graded) exact sequence:
\[ 0\to W_{n+1}\to A\otimes A_q^{!*}\otimes W_n\to A\otimes A_{q-1}^{!*}\otimes W_n \to \ldots \to A\otimes A_1^{!*}\otimes W_n \to A\otimes A_0^{!*}\otimes W_n \to W_n \to 0  \]
with differentials as in Lemma \ref{lemma: kosz compl}. This complexes can the be spliced together to obtain a graded projective resolution of $S=A_0$ which is also minimal since each term $A\otimes A_i^{!*}\otimes W_n$ is generated in degree $n(p+q)+i$ (see \cite[Remark 2.3]{yu}). 
This observation gives the following fact from \cite{yu}.

\begin{lemma} \label{lemma:cy alm kosz}
Let $A$ be a Frobenius $(p,q)$-Koszul algebra with $p>2$. If $A$ is stably Calabi-Yau of dimension $d$ then $q+1$ divides $d+1$.
\end{lemma}

\begin{proof}
See \cite[Proposition 4.4]{yu}.
\end{proof}

\begin{prop}
Let $A$ be a Frobenius $(p,q)$-Koszul algebra with $p>2$ and let $\eta$ and $\beta$ be Nakayama automorphisms of $A$ and $A^!$ respectively. Then $A$ is stably Calabi-Yau if and only if there exists a natural isomorphism $(\xi^{q+1}\bar{\beta})^*\cong \Omega_A^{-m}$ as functors in the stable module category, for some integer $m>0$. In this case $q+1$ divides $m$.
\end{prop}

\begin{proof}
By Corollary \ref{coroll: frob stabCY} $A$ is twisted stably Calabi-Yau of dimension $q$ with respect to the twist $\xi^{q+1}\bar{\beta}$, where $\beta$ is the Nakayama automorphism of $A^!$ and $\bar{\beta}$ is the automorphism of $A$ induced by $\beta$. 

Suppose that there is a natural isomorphism $(\xi^{q+1}\bar{\beta})^*\cong \Omega_A^{-m}$ for some $m>0$. Then we also have the following natural isomorphism:
\[ \nu\cong \Omega_A^{-q-1}\circ\Omega_A^{-m}\cong\Omega_A^{-q-1-m}, \]
therefore $A$ is stably Calabi-Yau of dimension $q+m$. Applying Lemma \ref{lemma:cy alm kosz} we have that $q+1$ divides $q+m+1$ and so $m=(k-1)(q+1)$ for some $k>0$.

Conversely assume that $A$ is stably Calabi-Yau of dimension $d$, so that we have a natural isomorphism $\nu\cong\Omega_A^{-d-1}$. The by Lemma \ref{lemma:cy alm kosz} there is some $k>0$ such that $d+1=k(q+1)$. Since $A$ is also twisted stably Calabi-Yau we obtain the following natural isomorphism:
\[ \Omega_A^{-(q+1)}\circ\Omega_A^{-(q+1)(k-1)}\cong \Omega_A^{-(q+1)}\circ(\xi^{q+1}\bar{\beta})^*. \]

Therefore, since $\Omega_A$ is an autoequivalence on the stable module category we obtain the desired isomorphism for $m=(k-1)(q+1)$.
\end{proof}

\section{Bimodule stably Calabi-Yau properties}\label{sec :bimod}

In this section we recall the definition of graded \emph{bimodule stably Calabi-Yau} algebras from \cite{AmOp} ad we introduce a twisted version of such property. Analogously to the twisted case this property implies that the stable category of \emph{Cohen-Macaulay} (one-sided) modules is Calabi-Yau. In particular we will study such properties for derivation algebras and for this pourpose we will need two different graded structures on our algebras. Therefore we start this section introducing these structures.

\subsection{\texorpdfstring{$\mathbb{Z}^2$}{double graded}-grading on derivation algebras} 

Assume that the semisimple algebra $S$ is graded and concentrated in degree $0$ and assume also that $V=\oplus_{i\geq 0}V_i$ is a positively graded $S$-$S$-bimodule with finite dimensional graded parts. Reacll that we denote by $S\lgr$ the category of left graded $S$-modules with morphisms in degree zero. The iterated tensor products $V^{\otimes j}$ also inherit the grading using distributivity of $\otimes$ and $\oplus$. Then the tensor algebra $T_S(V)$ has a $\mb{Z}^2$-grading:
\[T_S(V)=\bigoplus_{i,j\geq 0}T_S(V)_{i,j} \] 
in which the first part is the tensor grading and the second part is the grading obtained by the natural grading on $V^{\otimes j}$ induced by the grading on $V$. We call this grading the $V$-grading.
Under the choice of a (degree-zero) trace function $tr:S\to \Bbbk$, the graded versions of the dualities considered at the beginning of Subsection \ref{section: deriv alg} are again naturally isomorphic. Moreover the dual bimodule $(V^*)^{\otimes j}$ is negatively graded for any $j\geq 0$.

\begin{defin}\label{defin: cut}
We say that a superpotential $\omega$ in $T_S(V)$  \emph{admits a cut} if $\omega$ is homogeneous of $\mb{Z}^2$-degree $(d,1)$ for some $d>0$. That is, $\omega$ is homogeneous of degree $d$ with respect to the tensor grading (as in the classical Definition \ref{defin:superpotential}) and homogeneous of degree $1$ with respect to the $V$-grading.
\end{defin} 

\begin{rmk}
Our notation takes inspiration from the classical setting of quiver with potential (\cite{HerIy}) where a potential is a linear combination of cycles in the quiver $Q$ (up to cyclic permutation) and a \emph{cut} $C$ is a set of arrows such that each cycle in $Q$ contains precisely one arrow $\alpha\in C$.
\end{rmk} 

The evaluation maps $V^*\otimes V\to S$ and $V\otimes V^*\to S$ defining the pairing $[-]$ (recall the construction from Section \ref{section: deriv alg}) are zero on components $V^*_i\otimes V_j$ for $i\neq j$, therefore they are defined only on the homogeneous components of degree zero $V^*_i\otimes V_i$. The evaluation maps extend to a graded pairing of degree zero on the tensor product $(V^*)^{\otimes k}\otimes V^{\otimes l}\to V^{\otimes l-k}$ and $V^{\otimes l}\otimes(V^*)^{\otimes k} \to V^{\otimes l-k}$, analogously to the ungraded case. Recall that we have the following maps:
\[ \Delta_k^\omega:(V^{\otimes k})^*\otimes S_\varphi\to V^{\otimes d-k}, \quad \Delta_k^\omega(f\otimes s)=[f\omega s] \]
for any $0\leq k\leq d$.

\begin{lemma}\label{lemma:bigrad W}
If $\omega$ is a superpotential admitting a cut then $W_{d-k}=\im(\Delta_k^\omega)$ is a $\mb{Z}^2$-graded subspace of $V^{\otimes d-k}$ concentrated in $V$-degree $0$ and $1$.
\end{lemma}

\begin{proof}
Suppose that $\omega$ is homogenous of degree $(d,1)$. The maps 
\[ \Delta_k^\omega:(V^{\otimes k})^*\otimes S_\varphi\to V^{\otimes d-k}, \quad \Delta_k^\omega(f\otimes s)=[f\omega s] \]
are graded of degree $(d,1)$ hence the image $W_{d-k}=\im(\Delta_k^\omega)$ is a $\mb{Z}^2$-graded subspace of $V^{\otimes d-k}$. Moreover, since elements of $W_{d-k}$ are obtained by evaluating $k$ components of the linear terms of $\omega$, they are of $V$-degree at most $1$. \end{proof}

Note that $W_i$ is homogeneous of degree $i$ with respect to the first grading but it is not homogeneous with respect to the $V$-grading in general. Nevertheless we have a well defined $\mb{Z}^2$-grading on the derivation algebra $A=\mc{D}_\varphi(\omega, d-k)=T_S(V)/\langle W_k\rangle$ for any graded automorphism $\varphi$, with bigraded decomposition $A=\bigoplus_{i,j\geq 0}A_{i,j}$. 
\begin{notation}
From now on we will assume that $\omega$ admits a cut and w will consider this $\mb{Z}^2$-graded structure on $A$.\end{notation}

\begin{lemma}\label{lemma: bigraded complex}
The complex $\mc{W}^\bullet$ associated to $A$ is a complex of $\mb{Z}^2$-graded projective bimodules $A\otimes W_i\otimes A$ which are generated in degree $(i,0)$ and $(i,1)$, with differentials in degree $(0,0)$.

Moreover there are graded isomorphisms $(W_i)^*\cong(W_{d-i})_\varphi\langle -d,-1\rangle$ inducing a graded duality of $A$-bimodules $\Hom_{A^e\lgr}(A\otimes W_{i}\otimes A,A^e) \cong A\otimes W_{d-i}\otimes A_\varphi\langle -d,-1\rangle$ such that $\Hom_{A^e\lgr}(D_i,A^e)= D_{d+1-i}\langle -d,-1\rangle$.
\end{lemma}
\begin{proof}
The first statement is a consequence of Lemma \ref{lemma:bigrad W}.

The second statement is the graded version of Lemma \ref{lemma:selfdual complex} and the degree of the differentials is clearly zero. The degree shift of $\langle -d,-1\rangle$ comes from the fact that the duality $\Phi:(W_i)^*\xrightarrow{\simeq}(W_{d-i})_\varphi$ induced by the pairing
\[ [\phantom{a} ,\phantom{i} ] :(W_{d-i})^*\otimes (W_i)^*\to S,\quad [ f,g]:=[(f\otimes g)\omega] \]
has degree $|\omega|=(d,1)$ so it becomes an isomorphism in degree zero into the shifted module $(W_{d-i})_\varphi\langle -d, -1\rangle$.
\end{proof}

\begin{prop}\label{prop: kosz deg zero alg}
Let $A=\mc{D}_\varphi(\omega,d-2)$ be a $\mb{Z}^2$-graded algebra such that the complex $\mc{W}^\bullet$ is exact except in degree zero, where $H^0(\mc{W}^\bullet)\cong A$, and in degree $d$. Let $\Lambda$ be the $V$-degree zero subalgebra of $A$. Then $\Lambda$ is Koszul and $\gldim\Lambda\leq d-1$.
\end{prop}

\begin{proof}
By assumption we have the following $\mb{Z}^2$-graded exact complex
\begin{equation} \label{eq: compl 2}
A\otimes W_{d}\otimes A \xrightarrow{D_{d}} \cdots \xrightarrow{D_{2}}A\otimes W_{1}\otimes A  \xrightarrow{D_1} A\otimes W_0 \otimes A\to A \to 0, 
\end{equation}
with differentials of degree $(0,0)$ and where the projective bimodules $A\otimes W_i\otimes A$ generated in degree $(i,0)$ and $(i,1)$. In particular the bimodule $A\otimes W_d\otimes A$ is generated in degree $(d,1)$ since $W_d=\langle \omega \rangle$ and $\omega$ admits a cut. By Lemma \ref{lemma:one sided complexes}, applying the functor $-\otimes S$ to the complex (\ref{eq: compl 2}) we get an exact complex of $\mb{Z}^2$-graded left projective modules generated in degrees as before:
\begin{equation}\label{eq: compl 1} 
 A\otimes W_{d} \xrightarrow{d_{d}} \cdots \xrightarrow{d_{2}}A\otimes W_{1}  \xrightarrow{d_1} A \to S \to 0.
 \end{equation}
 
To simplify the notation let $P_i:=A\otimes W_i$. Since $\Lambda$ is the $V$-degree zero subalgebra of $A$, the $V$-degree zero part of $P_i$, that we denote by $(P_i)_0$, is a projective left $\Lambda$-module and the degree zero part of $\Lambda$ with respect to the tensor degree coincides with $S$ viewed as a $\Lambda$-module. Note moreover that $(P_d)_0=0$ since $P_d=A\otimes W_d$ is generated in $V$-degree $1$. Then, taking the $V$-degree zero part of complex (\ref{eq: compl 1}), we get an exact sequence:
\[ 0\to(P_{d-1})_0 \to \cdots \to (P_2)_0 \xrightarrow{\overline{d}_2} (P_1)_0\ \xrightarrow{\overline{d}_1} (P_0)_0\to {}_\Lambda S\to 0,\]
 of projective graded left $\Lambda$-modules (hence they are all concentrated in $V$-degree zero), each of which is generated in tensor degree $i$ and where the maps $\overline{d}_i$ are are morphisms in tensor degree zero. This gives a graded projective resolution of the (tensor) degree zero subalgebra of $\Lambda$, ${}_\Lambda S$, in which each projective module $(P_i)_0$ is generated in degree $i$ and so $\Lambda$ is Koszul of global dimension at most $d-1$.
\end{proof}

\subsection{Graded Cohen-Macaulay modules}

We recall the definitions and basic facts about Gorenstein algebras and Cohen-Macaulay modules.

\begin{defin}\label{defin:gorenst alg}
A finite dimensional algebra $A$ is said to be \emph{Gorenstein} it its injective dimension $\idim A$ and the projective dimension of its dual $\pdim A^*$ are both finite. If this is the case the natural number $\idim A=\pdim A^*$ is called the \emph{Gorenstein dimension} of $A$.
\end{defin}

\begin{defin}\label{defin:cm modules}
Let $A$ be a finite dimensional algebra. The category of \emph{(maximal) Cohen-Macaulay} $A$-modules is the full subcategory of $A\lmod$ with objects:
\[ \CM(A):=\left\lbrace M\in A\lmod \mid \ext_A^i(X,A)=0 \text{ for } i>0 \right\rbrace. \]

If $A$ is a graded algebra we denote by
\[ \CM\lgr(A):=\left\lbrace M\in A\lgr \mid \ext_A^i(X,A)=0 \text{ for } i>0 \right\rbrace \]
the category of graded Cohen-Macaulay left $A$-modules.
\end{defin}

The following result is a graded version of a well known result (see e.g. \cite{rickard1989derived}).

\begin{thm}\label{thm:cm stab}
If $A$ is a graded Gorenstein algebra, then $\CM\lgr(A)$ is a Frobenius category and there is a triangle equivalence
\[ \mc{D}^b(A\lgr)/\perf(A\lgr) \xrightarrow{\simeq} \underline{\CM\lgr}(A).  \]
\end{thm}

\begin{lemma}\label{lemma: CM prop}
Let $A$ be a graded Gorenstein algebra of dimension $g$ and $M\in A\lgr$. Then the following hold:
\begin{enumerate}
\item $\Omega_A^g(M)$ is Cohen-Macaulay.
\item There exists a short exact sequance 
\[0\to K\to {}_{\CM}M\to M\to 0 \]
such that $\pdim K<\infty$ and ${}_{\CM}M\in\CM\lgr(A)$. The module ${}_{\CM}M$ is called a ``Cohen-Macaulay replacement" of $M$ and can be chosen to only have projective direct summands that are generated in degree zero. 
\end{enumerate}
\end{lemma}

\begin{proof}
Statements number (1) and (2) are Lemmas 2.5 and 2.6 of \cite{AmOp} respectively (see also Remark 2.7 in the same article).
\end{proof}

\subsection{Bimodule stably Calabi-Yau algebras}

In this section we prove some sufficient conditions for a twisted quadratic derivation algebra $A=\mc{D}(\omega,|\omega|-2)$ to be \emph{bimodule stably Calabi-Yau}, following the definition given by Amiot and Oppermann in \cite{AmOp}.

\begin{defin}\label{defin: biodm st CY}
An algebra $A$ is called \emph{bimodule stably} $\varphi$-\emph{twisted d-Calabi-Yau} if there is an isomorphism
\[\mathbf{R}\Hom_{A^e}(A,A^e)[d+1]\cong A_\varphi  \text{ in }\underline{\CM}(A^e),\]
for some automorphism $\varphi$ of $A$.

Assume $A$ is a graded algebra. Then we say that $A$ is \emph{bimodule stably} $\varphi$-\emph{twisted d-Calabi-Yau} if there is an isomorphism
\[\mathbf{R}\Hom_{A^e}(A,A^e)[d+1]\cong A_\varphi\langle -p\rangle  \text{ in }\underline{\CM\lgr}(A^e),\]
for some natural number $p$ called the \emph{Gorenstein parameter}.
\end{defin}

\begin{prop}\label{prop: bimod stronger}
Let $A$ be a finite dimensional Gorenstein algebra. If $A$ is bimodule stably twisted $d$-Calabi-Yau then $A$ is stably twisted $d$-Calabi-Yau.
\end{prop}

\begin{proof}
This is the twisted version of Theorem 2.12 in \cite{AmOp} and the original proof applies in our setting. We repeat the argument here for the convenience of the reader.

First of all we note that if $A$ is Gorenstein then $A^e$ is also Gorenstein and therefore it has finite injective dimension. Moreover since we have the following bimodules isomorphisms:
\[ \Hom_{A^e}(A,(A^e)^*)\cong \Hom_\Bbbk(A\otimes_{A^e}A^e,\Bbbk)\cong A^*, \]
the complex $\mathbf{R}\Hom_{A^e}(A,A^e)$ is perfect. We can hence apply the first part of \cite[Lemma 4.1]{kellerCY} which says that there is a functorial isomorphism:
\[ D\Hom_{\mc{D}^b(A\lgr)}(X,Y)\cong \Hom_{\mc{D}^b(A\lgr)}(\mathbf{R}\Hom_{A^e}(A,A^e)\otimes^{\mathbf{L}}_A Y,X), \]
for any $X,Y\in\mc{D}^b(A\lgr)$.

It follows that the functor $\mathbf{R}\Hom_{A^e}(A,A^e)\otimes^{\mathbf{L}}_A-$ sends perfect complexes to perfect complexes so by \cite[Proposition 4.3.1]{amiotphd} the stable category of Cohen--Macaulay modules $\underline{\CM\lgr}(A)\cong\mc{D}^b(A\lgr)/\perf(A\lgr)$ has a Serre functor whose inverse is $\mathbf{R}\Hom_{A^e}(A,A^e)[1]\otimes^{\mathbf{L}}_A-$.

For any $X,Y\in \CM\lgr(A)$ there are functorial isomorphisms:

\[ \begin{split} D\pshom_{A\lgr}(X,Y) & \cong  \pshom_{A\lgr}(\mathbf{R}\Hom_{A^e}(A,A^e)[1]\otimes^{\mathbf{L}}_A Y,X) \\
& \cong \pshom_{A\lgr}(A_\varphi[-d]\otimes_A Y,X) \\
 & \cong \pshom_{A\lgr}(Y,{}_{\varphi^{-1}}X[d]) \end{split} \]

\end{proof}

\begin{lemma}\label{lemma:CY cond equiv}
Let $A$ be a graded algebra of Gorenstein dimension $< d/2$ with finite dimensional homogeneous parts $A_n$, for any $n\geq 0$. Then $A$ is bimodule stably $\varphi$-twisted $d$-Calabi-Yau with Gorenstein parameter $p$ if and only if the following condition holds in $\CM\lgr(A^e)$:

\begin{itemize} 
\item $\Hom_{A^e}\left(\Omega^{\frac{d+1}{2}}_{A^e}A,A^e\right)\cong\Omega^{\frac{d+1}{2}}_{A^e}(A_\varphi)\langle -p\rangle$ if $d$ is odd 

\item $\Hom_{A^e}\left(\Omega^{\frac{d}{2}+1}_{A^e}A,A^e\right)\cong\Omega^{\frac{d}{2}}_{A^e}(A_\varphi)\langle -p\rangle$ if $d$ is even.  \end{itemize} 
\end{lemma}

\begin{proof}
We prove the statement for $d$ odd, the even case following in the same way. 

First note that since $A$ has Gorenstein dimension $< d/2$ by Lemma \ref{lemma: CM prop} the modules $\Omega^{k}_{A^e}(A)$ for $k\geq  d/2$ are Cohen--Macaulay and for $k> d/2$ they are Cohen--Macaulay without projective direct summands. We have the following isomorphisms:

\begin{align*} & \mathbf{R}\Hom_{A^e}(A,A^e)[d+1]\cong A_\varphi\langle -p\rangle & \text{ in }\underline{\CM\lgr}(A^e)\phantom{a}  \\ 
\Leftrightarrow &  \mathbf{R}\Hom_{A^e}(A\left[-(d+1)/2\right],A^e)\cong A_\varphi\left[-(d+1)/2\right]\langle -p\rangle & \text{ in }\underline{\CM\lgr}(A^e) \phantom{a} \\ 
\Leftrightarrow & \Hom_{A^e}\left(\Omega^{\frac{d+1}{2}}_{A^e}(A),A^e\right)\cong  \Omega^{\frac{d+1}{2}}_{A^e}(A_\varphi)\langle -p\rangle & \text{ in }\underline{\CM\lgr}(A^e) \phantom{a} \\  
\Leftrightarrow & \Hom_{A^e}\left(\Omega^{\frac{d+1}{2}}_{A^e}(A),A^e\right)\cong  \Omega^{\frac{d+1}{2}}_{A^e}( A_\varphi)\langle -p\rangle & \text{ in }\CM\lgr(A^e), \end{align*}
where we can consider the functor $\Hom_{A^e}(-,-)$ instead of $\mathbf{R}\Hom_{A^e}(-,-)$ since the modules $\Omega^{\frac{d+1}{2}}_{A^e}( A_\varphi)$ and $\Omega^{\frac{d+1}{2}}_{A^e}( A)$ are Cohen--Macaulay and the last equivalence holds since both $\Omega^{\frac{d+1}{2}}_{A^e}(A_\varphi) $ and $\mathbf{R}\Hom_{A^e}(\Omega^{\frac{d+1}{2}}_{A^e}(A),A^e)$ are Cohen--Macaulay without projective direct summands and under our assumptions the category $\CM\lgr(A^e)$ is Krull-Schmidt.
\end{proof}

We can now prove the following result which generalizes the argument used in Theorem 5.7 of \cite{AmOp}.

\begin{thm}\label{thm:main}
Let $A=\mc{D}_\varphi(\omega,d-2)$ be a $\mb{Z}^2$-graded derivation algebra admitting  cut, with $|\omega|=(d,1)$ and twist $\varphi$ such that the associated complex $\mc{W}^\bullet$ is exact everywhere apart from $H^0(\mc{W}^\bullet ) \cong A$ and in position $d$. Suppose moreover that $A$ is Gorenstein of dimension $<\dfrac{d}{2}$. Then $A$ is bimodule stably $\varphi$-twisted $d$-Calabi-Yau of Gorenstein parameter $(d,1)$.
\end{thm}

\begin{proof}
We will prove the statement assuming $d$ odd; the case for $d$ even follows applying the same argument. For the rest of the proof we will also denote $(-)^\vee=\Hom_{A^e}(-,A^e)$.

For any derivation algebra $A$ as above we know from Lemma \ref{lemma: bigraded complex} that the associated bigraded complex of bimodules $\mc{W}^\bullet$ is selfdual, i.e.
\[ (A\otimes W_i\otimes A_\varphi)^\vee\cong 
 A\otimes W_{d-i}\otimes A_\varphi\langle -d,-1\rangle \text{ and } (D_i)^\vee= D_{d-i+1}\langle -d,-1\rangle.\]


Moreover, since by assumption $\mc{W}^\bullet$ is exact apart in homological degrees $0$ and $d+1$, it is the beginning of a bimodule projective resolution of $A$ and
\[ \Omega_{A^e}^{\frac{d+1}{2}}(A)\cong \im D_{\frac{d+1}{2}}. \]

Then we have the following isomorphisms:

\begin{align*}  \left(\Omega_{A^e}^{\frac{d+1}{2}}(A)\right)^\vee \cong \left( \im D_{\frac{d+1}{2}}\right)^\vee & \cong \coker\left(\left( \coker D_{\frac{d+1}{2}} \right)^\vee \to \left(A\otimes W_{\frac{d-1}{2}}\otimes A \right)^\vee \right) \\
& \cong \coker\left(\left( \Ker \left(D_{\frac{d+1}{2}}\right)^\vee \right) \to \left(A\otimes W_{\frac{d-1}{2}}\otimes A \right)^\vee \right) \\
& \cong \coker\left( \Ker \left( D_{\frac{d+1}{2}}\langle -d,-1\rangle \right)\right) \to \left( A\otimes W_{\frac{d-1}{2}}\otimes A_\varphi \langle -d,-1\rangle\right) \\
& \cong \im D_{\frac{d+1}{2}}\langle -d,-1\rangle \cong\Omega_{A^e}^{\frac{d+1}{2}}(A_\varphi)\langle -d,-1\rangle. \end{align*}

Therefore $A$ is bimodule stably $\varphi$-twisted $d$-Calabi-Yau by Lemma \ref{lemma:CY cond equiv}. 
\end{proof}

\begin{rmk}
Note that if we forget the graded structure in Theorem \ref{thm:main} the rpoof is still true with the only difference that $A$ is bimodule stably Calabi-Yau as an ungraded algebra (recall Definition \ref{defin: biodm st CY}). Later on, in connection with preprojective algebras, we will be particularly interested in the second component of the Gorenstein parameter, so it can also be useful to forget only the tensor-graded structure on $A$.
\end{rmk}

\subsection{Frobenius derivation algebras}

We can apply Theorem \ref{thm:main} to study bimodule stably Calabi-Yau properties of Frobenius almost Koszul algebras. Under the assumptions of Yu's Theorem \ref{thm: 3.7} characterizing Frobenius almost Koszul algebras, the hypotesis of Theorem \ref{thm:main} are satisfied and we get the following result about bimodule stably Calabi-Yau properties of derivation algebras.

\begin{thm}
Let $A$ be a finite dimensional, Frobenius, $(p,q)$-Koszul algebra of periodic type, say $A\cong\mc{D}_{\xi^{q+1}\bar{\beta}}(\omega, q-2)$ where $|\omega|=q$ and twist determined up to inner automorphism. Then $A$ is bimodule stably $\xi^{q+1}\bar{\beta}$-twisted Calabi-Yau of Gorenstein parameter $q$. 
\end{thm}

\begin{proof}
The algebra $A$ is a twisted derivation algebra by Theorem \ref{thm: 3.7} with twist $\xi^{q+1}\bar{\beta}$ and for which $\mc{W}^\bullet$ gives the first $q$ terms of a bimodule projective resolution of $A$. It is moreover graded with respect of the tensor grading therefore we can apply the single graded version of Theorem \ref{thm:main} and $A$ is bimodule stably twisted Calabi-Yau of parameter $q$.
\end{proof}

\section{Higher preprojective algebras coming from superpotentials}\label{sec: preproj}

The aim of this section is to prove a partial converse of Theorem \ref{thm: GI1} (i.e. \cite[Corollary 4.3]{GI}). We want to show that any selfinjective (non-twisted) derivation algebra satisfying the assumptions of Theorem \ref{thm:main} is isomorphic to the higher preprojective algebra of a Koszul, $n$-representation finite algebra. Moreover this algebra is the degree-zero part of the derivation algebra with respect to a certain grading making the derivation algebra stably bimodule Calabi-Yau of Gorenstein parameter one.

\subsection{Higher preprojective algebras}
First of all let us recall the construction of higher preprojective algebras. For any finite dimensional algebra $A$ denote by
\[ \tau: A\psmod \to A\ismod \mbox{ and } \tau^{-}: A\ismod \to A\psmod \]
the Auslander--Reiten translation and its inverse respectively. When $A$ is hereditary then $\tau$ is an endofuctor of the module category and $\tau^{-1}$ is left adjoint to $\tau$. Iyama in \cite{Iyama200722} generalized this results for algebras with higher global dimension. Let $d\geq 1$, the $d$-Auslander--Reiten translation and its inverse are defined as 
\[ \tau_d:=\tau\Omega^{d-1}: A\psmod \to A\ismod \mbox{ and } \tau^{-1}_d:=\tau^{-}\Omega^{-(d-1)}: A\ismod \to A\psmod. \]

If $\gldim A\leq d$ then $\tau_d$ and $\tau^-_d$ are the endofuctors:
\[ \tau_d=\ext^d_A(-,A)^*: A\lmod \to A\lmod \mbox{ and } \tau_d^{-1}=\ext^d(A^*,-): A\lmod \to A\lmod. \]

From now on we will assume that $\gldim A=d$ and we will denote $E:=\ext^d_A(A^*,A) \cong \ext^d_{A^e}(A,A^e)\cong \ext^d_{A^{op}}(A^*,A)\in A^e\lmod$ (the bimodule isomoprhisms are proved for instance in \cite[Lemma 2.9]{GI}). The following description of $\tau_d$ and $\tau^-_d$ can be found in \cite[Lemma 2.10]{GI} (see also \cite[Lemma 2.13]{IO2}).

\begin{prop}\label{prop:AR trans}
If $\gldim A\leq d$ we have the following isomorphisms of functors:
\[ \tau_d\cong\Hom_A(E,-): A\lmod \to A\lmod \mbox{ and } \tau_d^{-1}\cong E\otimes_A-: A\lmod \to A\lmod. \]
In particular $\tau^{-1}_d$ is left adjoint to $\tau_d$.
\end{prop}

\begin{defin}[\cite{n-aprIyama}]\label{def: hpreproj alg}
Let $A$ e a finite dimensional algebra of global dimension at most $d$. The $(d+1)$-\emph{preprojective algebra} of $A$ is the tensor algebra of $E$ over $A$:
\[ \Pi=\Pi_{d+1}(A):= T_A(E). \]
\end{defin}

Note that being a tensor algebra, $\Pi$ is naturally a graded algebra and we will refer to the grading as the \emph{tensor grading}. The $i$th part of the tensor grading is $E^{\otimes i}\cong \Hom_A(A,\tau_d^{-i}(A))$. This grading is in general coarser than the radical grading since $A$ is not semisimple.

If we start assuming that the algebra $A$ is graded (and this is the case for instance in the setting of Theorem \ref{thm: GI1}), then this grading induces a graded structure on the $A$-$A$-bimodule $E$ and, in turn, on the tensor algebra $\Pi=T_A(E)$. Hence $\Pi$ comes endowed two gradings and their total grading coincides with the radical grading of $\Pi$ (see \cite[Subsection 3.4]{GI} for a detailed presentation). Here we want to consider such a graded setting in the case of derivation algebras.

Before we can prove our result we need the following Theorem originally proved by Amiot and Oppermann in \cite{AmOp}.

\begin{thm}[\cite{AmOp}, Theorem 4.8]\label{thm:Am op}
Let $\Pi$ be a finite dimesional selfinjective positively graded algebra which is bimodule stably $d$-Calabi-Yau of Gorenstein parameter one. Then $\Lambda=\Pi_0$ is a $(d-1)$-representation finite algebra and there is an isomorphism of graded algebras $\Pi\cong\Pi_d(\Lambda)$.
\end{thm}

Now we are ready to prove the main result of this section.

\begin{thm}\label{thm: deriv alg preproj}
Let $A=\mc{D}(\omega,d-2)$ be a finite dimensional $\mb{Z}^2$-graded derivation algebra admitting a cut such that $|\omega|=(d,1)$ and such that $\mc{W}^\bullet$ is exact except in degree zero, where $H^0(\mc{W}^\bullet)\cong A$, and in degree $d$. Let $\Lambda:=(A)_{-,0}=\bigoplus_{i\geq 0}A_{i,0}$ be the $V$-degree zero subalgebra of $A$. Then $\Lambda$ is Koszul, $(d-1)$-representation finite and we have an isomorphism $\Pi_d(\Lambda)\cong A$ of graded algebras with respect to the $V$-grading on $A$.
\end{thm}

\begin{proof}
First of all, since $A$ is finite dimensional and with associated bimodule complex $\mc{W}^\bullet$ exact except in the first and last degree, by Theorem \ref{thm:self derivat alg} $A$ is selfinjective. Then we can apply the untwisted version of Theorem \ref{thm:main} to deduce that $A$ is bimodule stably $d$-Calabi-Yau of Gorenstein parameter $(d,1)$. Forgetting the tensor-grading we can see $A$ as bimodule stably Calabi-Yau of parameter one.

By Proposition \ref{prop: kosz deg zero alg} the $V$-degree zero subalgebra $\Lambda=\bigoplus_{i\geq 0}A_{i,0}$ is Koszul and $\gldim\Lambda\leq d-1$. Moreover by Teorem \ref{thm:Am op} $\Lambda$ is $(d-1)$-representation finite and $A\cong\Pi_d(\Lambda)$ as a graded algebra with respect to the $V$-grading. 
\end{proof}



\bibliographystyle{alpha}

\bibliography{bibliogabri.bib}{}


\end{document}